\documentclass[11pt,intlimits, draft,oneside]{amsart}

\usepackage[latin2]{inputenc}
\usepackage{latexsym}
\usepackage{amssymb}
\usepackage{amsthm}

\usepackage{amsmath}

\usepackage{enumerate}

\input{cyracc.def}
\newfont{\cyrfnt}{wncyi10 at 11pt}

\newtheorem{thm}{Theorem}[section]

\newtheorem{lemma}[thm]{Lemma}

\theoremstyle{definition}

\newtheorem{remark}[thm]{Remark}

\newcommand{\R}{\mathbb{R}}

\newcommand{\N}{\mathbb{N}}

\newcommand{\UU}{\mathcal{U}}
\newcommand{\II}{\mathcal{I}}
\newcommand{\CC}{\mathcal{C}}
\newcommand{\dist}{\text{dist}}

\renewcommand{\Re}{\mathrm{Re}\,}

\def\rg{\mathop{\mathrm{rg}}}

\begin{document}
\title[A ``typical''  contraction is unitary]
{A ``typical''  contraction is unitary
}

\author{Tanja Eisner}
\address{Mathematisches Institut,
  Universit\"{a}t T\"{u}bingen \newline 
  Auf der Morgenstelle 10, D-70176, T\"{u}bingen, Germany}
\email{talo@fa.uni-tuebingen.de}


\begin{abstract}
We show that (for the weak operator topology) the set of unitary operators on a separable
infinite-dimensional Hilbert space is residual in the set of all
contractions. The analogous
result holds for isometries and the strong operator topology as
well. These results are applied to 
the problem of embedding operators into
strongly continuous semigroups.     
\end{abstract}
\maketitle

\section{Introduction}


Unitary operators are the nicest and best understood operators on
Hilbert spaces, 
and there are various results showing that there are ``sufficiently many''
unitary operators among the contractions. For
example, the theory of Foia{\c s}--Sz. Nagy extends every contraction
to a unitary operator, called ``unitary dilation'', on a larger
space. Furthermore, Peller \cite{peller:1981} showed that 
the set of all unitary operators is dense in the set of all
contractions with respect to the weak operator topology (and even for
some stronger operator topology called power-weak (pw) operator topology).
Finally, operator functions are bounded from above by their value on unitary operators, as proved by Nelson \cite{nelson:1961}.

In this note we improve the above density result of Peller and show
that on a (separable infinite-dimensional) Hilbert space the unitary
operators form a residual subset of the set of all contractions
(isometries) with respect to the weak (strong) operator topology, see
Section 1 and 2. (Recall that a set of a Baire space is called residual if its complement is of first category.) In this sense, a ``typical'' contraction or isometry is unitary.

We apply this result to the problem of embedding an operator into a
strongly continuous semigroup, see Section 3. We show that a
``typical'' contraction or isometry on a separable
infinite-dimensional Hilbert space is embeddable. This is an
operator-theoretic counterpart to a recent result from ergodic theory
on embedding a measure-preserving transformation into a flow, see de
la Rue, de Sam Lazaro \cite{delarue/desamlazaro:2003}. In particular, a ``typical'' contraction has roots of all order, extending the analogous result from ergodic theory, see King \cite{king:2000}.    

\vspace{0.1cm}

\section{Isometries}

Let $H$ be an infinite-dimensional separable Hilbert space. 
We denote by $\UU$ the set of all unitary operators and by $\II$ the
set of all isometries on $H$ endowed with the strong operator
topology. We show in this section that $\UU$ is residual in $\II$,
i.e., its complement $\II\setminus \UU$ is of first category. This shows that a ``typical'' isometry is unitary.

The space $\mathcal{I}$ is a complete metric space with respect to the metric  
\begin{equation*}
d(T,S):= \sum_{j=1}^\infty \frac{\|Tx_j -Sx_j\|}{2^j \|x_j\|}\quad \text{for } T,S \in \mathcal{I},
\end{equation*}
where $\{x_j\}_{j=1}^\infty$ is a fixed dense subset of $H\setminus \{0\}$. 
    
The starting point is the following lemma based on the Wold decomposition of an
isometry, see e.g. Eisner, Sereny \cite{eisner/sereny:2008}.
\begin{lemma}\label{lemma:unitary-dense-isometry}
The set $\UU$ of unitary operators is dense in $\II$. 
\end{lemma}
However, the following much stronger result holds.
\begin{thm}\label{thm:typical-isometry}
The set $\UU$ of unitary operators is residual in $\II$.
\end{thm}
\begin{proof}
Let $T$ be a non-invertible isometry. Then $\rg T$ is closed and
different from $H$. Therefore, there exists $x_j$ with $\dist(x_j, \rg
T)>0$, hence 
\begin{equation*}
 \II\setminus \UU= \bigcup_{k,j=1}^\infty M_{j,k} \quad \text{ with }\quad M_{i,k}:=\left\{T: \dist(x_j,\rg T)>\frac{1}{k}\right\}.
\end{equation*}

We prove now that every set $M_{j,k}$ is nowhere dense in $\II$. By Lemma \ref{lemma:unitary-dense-isometry} it suffices to show that
\begin{equation}\label{eq:nowhere-dense}
\UU\cap \overline{M_{j,k}}=\emptyset \quad \forall j,k. 
\end{equation}
Assume the contrary, i.e., that there exists a sequence
$\{T_n\}_{n=1}^\infty\subset M_{j,k}$ for some $j,k$ and a unitary
operator $U$ with $\lim_{n\to\infty}T_n=U$ strongly. In particular,
$\lim_{n\to\infty} T_n y=Uy=x_j$ for $y:=U^{-1}x_j$. This however
implies $\lim_{n\to \infty}\dist(x_j, \rg T_n)=0$, a contradiction. So
(\ref{eq:nowhere-dense}) is proved, every set $M_{j,k}$ is nowhere dense, and $\UU$ is residual in $\II$.
\end{proof}

  \vspace{0.1cm}

\section{Contractions}

As before, we consider a separable infinite-dimensionals Hilbert space $H$
and prove now that a ``typical'' contraction on $H$ is unitary. 

Denote by $\CC$ the set of all contractions on $H$ endowed with the weak operator topology. This is a complete metric space with respect to the metric 
\begin{equation*}
d(T,S):= \sum_{i,j=1}^\infty \frac{|\left<Tx_i,x_j\right> -\left< Sx_i,x_j\right>|}{2^{i+j} \|x_i\| \|x_j\|}\quad \text{for }\ T,S \in \mathcal{C},
\end{equation*}
where $\{x_j\}_{j=1}^\infty$ is a fixed dense subset of $H\setminus \{0\}$. 

The following density result holds for unitary operators, see Peller \cite{peller:1981} (for a much stronger result) or Eisner, Ser\'eny \cite{eisner/sereny:2008}.   
\begin{lemma}\label{lemma:unitary-dense-contr}
The set $\UU$ of unitary operators is dense in $\CC$. 
\end{lemma}
Our construction uses the following well-known property of weak
convergence, see e.g. Halmos \cite{halmos:1967}. For the reader's
convenience we give its simple proof.
\begin{lemma}\label{w-imply-st}
Let $\{T_n\}_{n=1}^\infty$ be a sequence of linear operators on a Hilbert space $H$ converging weakly to a linear operator $S$. If $\|T_n x\|\leq \|Sx\|$ for every $x\in H$, then $\displaystyle \lim_{n\to\infty}T_n=S$ strongly.
\end{lemma}
\begin{proof}
For every $x\in H$ we have
\begin{eqnarray*}
\Vert T_n x - Sx \Vert^2 
 &=& \left< T_n x - Sx, T_n x - Sx \right> = \Vert Sx \Vert^2 + \Vert T_n x \Vert^2  - 2 \Re \left< T_n  x, S x \right> \\ 
 &\leq& 2 \left< Sx,Sx \right>  - 2\Re\left< T_n x, Sx \right> = 2 \Re \left< (S - T_n) x, Sx \right> \to 0 \ \text{  as  } \ n\to\infty, 
\end{eqnarray*}
and the lemma is proved.
\end{proof}

We now state the main result of this paper. 
\begin{thm}\label{thm:typical-contr}
The set $\UU$ of unitary operators is residual in the set $\CC$ of contractions. 
\end{thm}
\begin{proof}
We first prove that the set $\II\setminus \UU$ of non-invertible isometries is of first category. 
As in the proof of Theorem \ref{thm:typical-isometry},
$\II\setminus\UU$ is given as  
\begin{equation*}
 \II\setminus \UU= \bigcup_{k,j=1}^\infty M_{j,k} \quad \text{ with }\quad M_{i,k}:=\left\{T \text{ isometric}: \dist(x_j,\rg T)>\frac{1}{k}\right\}.
\end{equation*}
By Lemma \ref{lemma:unitary-dense-contr} it is enough to show that 
\begin{equation*}
\UU\cap \overline{M_{j,k}}=\emptyset \quad \forall j,k. 
\end{equation*}
Assume that for some $j,k$ there exists a sequence
$\{T_n\}_{n=1}^\infty\subset M_{j,k}$ converging weakly to a unitary
operator $U$. Then, by Lemma \ref{w-imply-st}, $T_n$ converges to $U$
strongly. As in the proof of  Theorem \ref{thm:typical-isometry}, this
implies $\lim_{n\to\infty}T_ny=Uy=x_j$ for $y:=U^{-1}x_j$. Hence
$\lim_{n\to\infty}\dist(x_j,\rg T_n)=0$  contradicting
$\{T_n\}_{n=1}^\infty\subset M_{j,k}$, so every $M_{j,k}$ is nowhere
dense and $\II\setminus \UU$ is of first category.

We now show that the set of non-isometric operators is of first category in $\CC$ as well. 
Let $T$ be a non-isometric contraction. Then there exists $x_j$ such
that $\|Tx_j\|<\|x_j\|$, hence
\begin{equation*}
 \CC\setminus \II= \bigcup_{k,j=1}^\infty N_{j,k} \quad \text{ with }\quad N_{j,k}:=\left\{T: \frac{\|Tx_j\|}{\|x_j\|}<1-\frac{1}{k}\right\}.
\end{equation*} 
It remains to show that every $N_{j,k}$ is nowhere dense in $\CC$. By
Lemma \ref{lemma:unitary-dense-contr} again it suffices to show 
\begin{equation*}
\UU\cap \overline{N_{j,k}}=\emptyset \quad \forall j,k. 
\end{equation*}
Assume that for some $j,k$ there exists a sequence
$\{T_n\}_{n=1}^\infty\subset N_{j,k}$ converging weakly to a unitary
operator $U$. Then $T_n$ converges to $U$ strongly by Lemma
\ref{w-imply-st}. This implies in particular that
$\lim_{n\to\infty}\|T_n x_j\|=\|Ux_j\|=\|x_j\|$ contradicting
$\frac{\|T_nx_j\|}{\|x_j\|}<1-\frac{1}{k}$ for every $n\in \N$.
\end{proof}


\begin{remark}
Peller \cite{peller:1981} showed that the set of unitary operators is
dense in the set of contractions with respect to the so-called
pw-topology (power-weak operator topology). This topology corresponds
to weak convergence of all powers, i.e.,
$$
\displaystyle
\text{pw-lim}\ T_n=S \quad 
\Longleftrightarrow
\quad
\lim_{n\to\infty}T_n^k=S^k \text{ weakly for all } k\in\N.
$$

It is natural to ask whether unitary operators are residual 
with respect to this topology as well. Indeed, all the arguments used in the proof of Theorem
\ref{thm:typical-contr} remain valid for this topology, hence  the complement
of $\UU$ is of first category. However, the space of all contractions
is no longer complete, see Eisner, Ser\'eny \cite{eisner/sereny:2008-tk}
for a diverging Cauchy sequence. So it is not clear whether $\CC$ with
the pw-topology is a Baire space. 
\end{remark}

 \vspace{0.1cm}

\section{Application to the embedding problem}

In this section we consider the following problem: Which bounded
operators $T$ can be \emph{embedded} into a strongly continuous
semigroup, i.e., does there exist a $C_0$-semigroup $(T(t))_{t\geq 0}$
with $T=T(1)$? (For the basic theory of $C_0$-semigroups we refer to Engel, Nagel \cite{engel/nagel:2000}.)
For some classes of operators this question has a positive answer, e.g., for operators with spectrum in a
certain area using functional calculus, see e.g. Haase \cite[Section 3.1]{haase:2006}, and for isometries on Hilbert spaces with infinite-dimensional kernel, see \cite{eisner:2008}.

We apply the above category results to this
problem and show that a ``typical'' isometry and a ``typical''
contraction on a separable infinite-dimensional Hilbert space is embeddable. 

It is well-known that unitary operators have the embedding property.
%
\begin{lemma}
Every unitary operator on a Hilbert space can be embedded into a unitary $C_0$-group.
\end{lemma}
\begin{proof}
Let $T$ be a unitary operator on a Hilbert
space $H$. Then by the spectral theorem, see
e.g. Halmos \cite{halmos:1963}, $T$ is isomorphic to a direct sum of
multiplication operators $M$ given
by $Mf(e^{i\varphi}):=e^{i\varphi}f(e^{i\varphi})$
on $L^2(\Gamma,\mu)$ for the unit circle $\Gamma$ and some
measure $\mu$. Each such operator can be embedded into the unitary $C_0$-group $(U(t))_{t\in\R}$ given by
$$
 U(t)f(e^{i\varphi}):=e^{it\varphi}f(e^{i\varphi}), \quad \varphi\in
 [0,2\pi],\ t\in\R. 
$$
\end{proof}
A direct corollary of Theorem \ref{thm:typical-isometry}, Theorem
\ref{thm:typical-contr} and the above lemma is the following category
result for embeddable operators. 
\begin{thm}
On a separable infinite-dimensional Hilbert space, the set 
of all embeddable operators is residual in the set $\II$ of all
isometries (for the strong operator topology) and in the set $\CC$ of all contractions (for the weak operator topology). 
\end{thm}
In other words, a ``typical'' isometry and a ``typical'' contraction can be embedded
into a $C_0$-semigroup. 
 This is an operator theoretic counterpart
to a recent result of de
la Rue, de Sam Lazaro \cite{delarue/desamlazaro:2003} in ergodic theory stating that a
``typical'' measure preserving transformation can be embedded into a
continuous measure preserving flow. 

\begin{remark}
In particular, a ``typical'' contraction (on a separable infinite-dimensional Hilbert
space) has roots of every order. This is an operator theoretic analogue
of a result of King \cite{king:2000} from ergodic
theory. See also Ageev \cite{ageev:2003} and Stepin, Eremenko \cite{stepin/eremenko:2004} for related results. 
\end{remark}


\vspace{0.15cm}

\noindent {\bf Acknowledgement.} 
The author is grateful to Rainer Nagel for valuable comments.

\parindent0pt


\begin{thebibliography}{10}

\bibitem{ageev:2003}
O.~N.~Ageev, %
\textit{On the genericity of some non-asymptotic dynamical properties}, %
Uspekhi Mat. Nauk  \textbf{58}  (2003),  177--178;  translation in 
Russian Math. Surveys \textbf{58} (2003), 173--174.
 

\bibitem{eisner:2008}
T.~Eisner, %
\textit{Embedding linear operators into strongly continuous
  semigroups},
preprint. 

\bibitem{eisner/sereny:2008}
{T.~Eisner} and {A.~Ser\'eny}, %
\textit{Category theorems for stable operators on {H}ilbert spaces}, %
Acta Sci. Math. (Szeged) \textbf{74} (2008), 259--270. %

\bibitem{eisner/sereny:2008-tk}
{T.~Eisner} and {A.~Ser\'eny}, %
\textit{On the weak analogue of the Trotter--Kato theorem}, submitted.


\bibitem{engel/nagel:2000}
K.-J. Engel and R.~Nagel, \emph{One-parameter {S}emigroups for
{L}inear
  {E}volution {E}quations}, Graduate Texts in Mathematics, vol. 194,
  Springer-Verlag, New York, 2000.


\bibitem{haase:2006}
M. Haase, %
\textit{The {F}unctional {C}alculus for {S}ectorial {O}perators}, %
Operator Theory: Advances and Applications, vol. 169, Birkh\"auser Verlag, Basel, 2006.


\bibitem{halmos:1963}
P.~R.~Halmos, %
\textit{What does the spectral theorem say?}, %
Amer. Math. Monthly \textbf{70} (1963), 241--247.


\bibitem{halmos:1967}
P.~R.~Halmos, \emph{A {H}ilbert {S}pace {P}roblem {B}ook},
  D. Van Nostrand Co., Inc., Princeton, N.J.-Toronto, Ont.-London 1967.


\bibitem{king:2000}
J.~L.~King, %
\textit{The generic transformation has roots of all orders}, %
Colloq. Math. \textbf{84/85} (2000), 521--547.


\bibitem{nagel:1986}
R.~Nagel~(ed.), \emph{One-parameter {S}emigroups of {P}ositive
{O}perators},
  Lecture Notes in Mathematics, vol. 1184, Springer-Verlag, Berlin, 1986.



\bibitem{nelson:1961}
E. Nelson, %
\emph{The distinguished boundary of the unit operator ball}, %
 Proc. Amer. Math. Soc.  \textbf{12}  (1961), 994--995. 


\bibitem{peller:1981} 
V.~V.~Peller, \emph{Estimates of operator polynomials in the space $L^p$ with respect to the multiplicative norm}, J. Math. Sciences {\bf 16} (1981), 1139--1149.


\bibitem{delarue/desamlazaro:2003}
T. de la Rue and J. de Sam Lazaro, %
\textit{The generic transformation can be embedded in a flow}
(French), %
Ann. Inst. H. Poincar\'e, Prob. Statist. \textbf{39} (2003),
121--134. 

\bibitem{stepin/eremenko:2004}
A.~M.~Stepin, A.~M. Eremenko, %
\textit{Nonuniqueness of an inclusion in a flow and the vastness of a centralizer for a generic measure-preserving transformation}, 
Mat. Sb. \textbf{195} (2004), 95--108; translation in
Sb. Math. \textbf{195} (2004), 1795--1808. 


\bibitem{sznagy/foias}
B.~Sz.-Nagy and C.~Foia{\c{s}}, %
\textit{Harmonic {A}nalysis of {O}perators on {H}ilbert {S}pace}, %
North-Holland Publishing Co., 1970. %


\end{thebibliography}
\end{document}